\pdfoutput=1
\documentclass[a4paper,12pt,reqno,oneside]{amsart}

\usepackage{geometry}
\usepackage{lmodern}
\usepackage[stretch=10]{microtype}

\usepackage{amsmath}
\usepackage{amssymb}
\usepackage{amsthm}
\usepackage[shortlabels]{enumitem}
\PassOptionsToPackage{hyphens}{url}
\usepackage[bookmarksnumbered,bookmarksdepth=2,bookmarksopen,colorlinks,linkcolor=black,citecolor=black,urlcolor=black,linktoc=all]{hyperref}
\usepackage[capitalise]{cleveref}
\usepackage[foot]{amsaddr}

\newcommand{\ZZ}{\mathbb{Z}}

\newcommand{\RR}{\mathbb{R}}

\newcommand{\GG}{\mathbb{G}}

\newcommand{\cO}{\mathcal{O}}

\newcommand{\fS}{\mathfrak{S}}

\newcommand{\fX}{\mathfrak{X}}

\DeclareMathOperator{\Frac}{Frac}

\DeclareMathOperator{\spmap}{sp}
\DeclareMathOperator{\Sp}{Sp}
\DeclareMathOperator{\Spf}{Spf}
\DeclareMathOperator{\Spec}{Spec}

\newcommand{\an}{\mathrm{an}}

\newcommand{\formal}{\mathrm{for}}

\newcommand{\red}{\mathrm{red}}

\newcommand{\abs}[1]{\lvert #1 \rvert}

\newcommand{\powerseries}[2]{#1 [\![ #2 ]\!]}
\newcommand{\tatealgebra}[2]{#1 \langle #2 \rangle}

\newcommand{\ov}{\overline}

\newcommand{\defterm}[1]{\textbf{#1}}

\newtheorem{lemma}{Lemma}[section]
\newtheorem{proposition}[lemma]{Proposition}
\newtheorem{theorem}[lemma]{Theorem}
\newtheorem{corollary}[lemma]{Corollary}

\Crefname{conjecture}{Conjecture}{Conjectures} % Work around bug in cleveref

\Crefname{claim}{Claim}{Claims}

\newtheorem*{lemma*}{Lemma}
\newtheorem*{proposition*}{Proposition}
\newtheorem*{theorem*}{Theorem}
\newtheorem*{corollary*}{Corollary}
\newtheorem*{claim*}{Claim}

\theoremstyle{definition}
\newtheorem*{definition}{Definition}
\newtheorem{notation}[lemma]{Notation}
\newtheorem{remark}[lemma]{Remark}

\usepackage{ifthen}
\newcounter{constant}

\newcommand{\newC}[1]{%
  \ifthenelse{\equal{#1}{*}} {%
      \stepcounter{constant} c_{\theconstant}%
  } {%
      \refstepcounter{constant} c_{\theconstant} \label{C:#1}%
  }%
}

\usepackage{color}

\title[Extension of relative rigid homomorphisms]{Extension of relative rigid homomorphisms from the formal multiplicative group}
\author{Martin Orr}

\address{Orr: Department of Mathematics, The University of Manchester, Oxford Road, Manchester M13 9PL, United Kingdom}
\email{martin.orr@manchester.ac.uk}


\begin{document}

\begin{abstract}
A theorem of L\"utkebohmert states that a rigid group homomorphism from the formal multiplicative group to a smooth commutative rigid group~$G$, with relatively compact image, can be extended to a homomorphism from the rigid multiplicative group to~$G$.
In this paper, we prove a relative version of this theorem over a geometrically reduced quasi-compact quasi-separated rigid space.
The relative theorem is proved under an additional hypothesis that some open relative subgroup of~$G$ has good reduction.
This theorem is useful for studying rigid uniformisation of abelian or semiabelian varieties in a relative setting.
\end{abstract}

\maketitle

\section{Introduction}

The theory of uniformisation of abeloid varieties (smooth proper rigid analytic groups) was developed by Raynaud \cite{Ray70}, Bosch \cite{Bos79} and Lütkebohmert \cite{Lut95}.  A key ingredient is an extension theorem for homomorphisms from the formal multiplicative group to smooth commutative rigid groups, established in a restricted form at \cite[Satz~2.2]{Bos79} and in general in the following theorem of Lütkebohmert.

\begin{proposition} \label{hom-extends-Lut} \cite[Prop.~3.1]{Lut95}
Let $R$ be a complete discrete valuation ring and let $K = \Frac(R)$.
Let $\ov\GG_m$ denote the rigid $K$-group $\{ x : \abs{x} = 1 \}$.

Let $G$ be a smooth connected commutative rigid group over~$K$.
Let $\bar\phi \colon \ov\GG_m \to G$ be a homomorphism of rigid groups.

If the image of $\bar\phi$ is relatively compact in some quasi-compact open subspace of~$G$, then there exists a unique homomorphism of rigid groups $\phi \colon \GG_m \to G$ whose restriction to $\ov\GG_m$ is equal to $\bar\phi$.
\end{proposition}

\pagebreak % COSMETIC

The aim of this article is to prove a relative version of \cref{hom-extends-Lut}, under an additional hypothesis stating that $G$ has an open (relative) subgroup with good reduction.
The main theorem of the article is as follows.
(For the definition of ``relatively $S$-compact'', see section~\ref{sec:relative-compactness}.)

\begin{theorem} \label{hom-extends-thm}
Let $S$ be a geometrically reduced quasi-compact quasi-separated rigid space over~$K$, with formal $R$-model $S_\formal$.
Let $G \to S$ be a smooth separated commutative rigid $S$-group with connected fibres.
Let $\bar\phi \colon \ov\GG_m \times S \to G$ be a homomorphism of rigid $S$-groups.
Suppose that:
\begin{enumerate}[label=(\roman*)]
\item $\bar\phi$ factors through an open $S$-subgroup $H \subset G$, such that $H \to S$ admits a smooth formal $R$-model $H_\formal \to S_\formal$ which is an $S_\formal$-group;
\item there is a quasi-compact open subspace $X$ of~$G$ such that the image of $\bar\phi$ is relatively $S$-compact in~$X$.
\end{enumerate}
Then there exists a unique homomorphism of rigid $S$-groups $\phi \colon \GG_m \times S \to G$ whose restriction to $\ov\GG_m \times S$ is equal to $\bar\phi$.
\end{theorem}

Condition~(i) is included in \cref{hom-extends-thm} because we do not yet have a relative version of \cite[Thm~1.2]{Lut95}, which gives the existence of a smooth formal group scheme generated by a morphism from a smooth formal scheme to a rigid group.

The author's motivation for writing this paper was his work with Daw on large Galois orbits in families of abelian varieties with multiplicative degeneration \cite{DO:LGO}, which requires a rigid uniformisation result for families of abelian varieties.
For this application, it is convenient to use the following corollary of \cref{hom-extends-thm}, which replaces the relative compactness condition~(ii) by the condition that the morphism $G \to S$ is without boundary.
See section~\ref{sec:without-boundary} for the definition of morphisms without boundary.
In particular, the following are examples of morphisms without boundary:
\begin{enumerate}
\item proper morphisms of rigid spaces;
\item rigid analytifications of morphisms of separated $K$-schemes of finite type.
\end{enumerate}

\begin{corollary} \label{hom-extends-no-boundary}
Let $S$ be a geometrically reduced quasi-compact quasi-separated rigid space over~$K$, with formal $R$-model $S_\formal$.
Let $G \to S$ be a smooth separated commutative rigid $S$-group with connected fibres.
Let $\bar\phi \colon \ov\GG_m^g \times S \to G$ be a homomorphism of rigid $S$-groups.
Suppose that:
\begin{enumerate}[label=(\roman*)]
\item $\bar\phi$ factors through an open $S$-subgroup $H \subset G$, such that $H \to S$ admits a smooth formal $R$-model $H_\formal \to S_\formal$ which is an $S_\formal$-group;
\item $G \to S$ is without boundary.
\end{enumerate}
Then there exists a unique homomorphism of rigid $S$-groups $\phi \colon \GG_m^g \times S \to G$ whose restriction to $\ov\GG_m^g \times S$ is equal to $\bar\phi$.
\end{corollary}

\subsection{Notes on the proof of Theorem~\ref{hom-extends-thm}}

The proof of \cref{hom-extends-thm} is based on the proof of \cref{hom-extends-Lut} found in \cite{Lut95} and \cite{Lut16}.
A key ingredient is Lütkebohmert's approximation theorem for rigid analytic morphisms \cite[Thm~7.4]{Lut95}.

In \cite{Lut95}, \cite[Thm.~1.2]{Lut95} is used to show that, in the situation of \cref{hom-extends-Lut}, condition~(i) of \cref{hom-extends-thm} is always satisfied.
It is also used (in the proofs of both \cite[Prop.~3.1]{Lut95} and \cite[Prop.~7.3.1]{Lut16}) to obtain a subgroup of~$G$ with smooth formal model with unipotent reduction.
Since we have not proved a relative version of \cite[Thm.~1.2]{Lut95}, we include condition~(i) in the statement of \cref{hom-extends-thm}.
We replace the second use of \cite[Thm~1.2]{Lut95} by the fact that the formal fibre of~$H$ at the identity element is an open ball (the relative version is \cite[Thm.~1.3.2]{Berthelot}) and calculations with power series showing that any closed ball in this fibre centered at the origin is itself a subgroup (\cref{formal-group-congruence-subgroup}).

The remainder of the proof follows that of \cite[Prop.~3.1]{Lut95}.

\subsection{Outline of the paper}

In section~\ref{sec:relative-compactness}, we recall the definition of a relatively $S$-compact subspace of a rigid space over~$S$, prove some elementary results about relatively $S$-compact subspaces, and recall the statement of the rigid Approximation Theorem.
In section~\ref{sec:group-laws-balls}, we prove some elementary results about the power series representing $S$-group laws on rigid balls over~$S$.
Section~\ref{sec:main-proof} contains the main proof of \cref{hom-extends-thm}.
Finally, in section~\ref{sec:without-boundary}, we recall the definition of morphisms of rigid spaces without boundary and prove \cref{hom-extends-no-boundary}.

\subsection{Notation}

Throughout the paper, $R$ denotes a complete discrete valuation ring, $K$ its field of fractions and $k$ its residue field.
% Let $\pi$ denote a uniformiser of~$R$.
Rigid spaces will always mean rigid spaces over~$K$ and formal schemes will always mean admissible formal schemes over $\Spf(R)$.

If $X_\formal$ is a formal scheme over~$\Spf(R)$, then $X_0$ denotes the $k$-scheme
\[ X_0 = X \times_{\Spf(R)} \Spec(k). \]
If $X_\formal$ is admissible and quasi-compact, then $X$ denotes Raynaud's rigid generic fibre of~$X_\formal$.

Let $\ov\GG_m$ denote the following open subgroup of the rigid $K$-group $\GG_m$:
\[ \ov\GG_m = \{ x : \abs{x} = 1 \}. \]

For $a \in K$, $r \in \RR_{>0}$, we use the following notation for rigid discs:
\begin{align*}
    D(a,r) & = \{ x : \abs{x-a} \leq r \},
  & D(r) & = D(0,r),
  & D & = D(0,1),
\\  D^+(a,r) & = \{ x : \abs{x-a} < r \},
  & D^+(r) & = D^+(0,r),
  & D^+ & = D^+(0,1).
\end{align*}

If $S$ is a reduced affinoid space, let $\cO(S)$ denote the ring of holomorphic functions on~$S$, with sup norm $\abs{\cdot}_S$.
Let
\[ \cO^\circ(S) = \{ f \in \cO(S) : \abs{f}_S \leq 1 \}. \]
Since we assume that $R$ is a discrete valuation ring, $\cO^\circ(S)$ is an $R$-algebra of topologically finite type \cite[Thm.~3.1.17]{Lut16}.

\subsection*{Acknowledgements}

I am very grateful to Werner Lütkebohmert for patient answers to my questions about relativising \cref{hom-extends-Lut}, in particular for suggesting that the argument in \cite{Lut95} should relativise, for helping me to understand the approximation theorem, and for suggesting how to prove \cref{relatively-compact-thickening}.
I am also grateful to Christopher Daw for helpful discussions.

\section{Relative \texorpdfstring{$S$}{S}-compactness} \label{sec:relative-compactness}

In this section, we prove several lemmas about relatively $S$-compact subspaces of rigid spaces, an essential notion in the statement and proof of \cref{hom-extends-thm}.  We also recall the statement of L\"utkebohmert's Approximation Theorem for rigid morphisms, another key ingredient in the proof of \cref{hom-extends-thm}.

\subsection{Definition of relatively \texorpdfstring{$S$}{S}-compact subspaces}

\begin{definition} \cite[Def.~3.6.1]{Lut16}
Let $X \to S$ be a morphism of quasi-compact, quasi-separated rigid spaces.
An open subspace $U \subset X$ is said to be \defterm{relatively $S$-compact}, written $U \Subset_S X$,
if there exists a formal $R$-model $X_\formal \to S_\formal$ of $X \to S$ such that:
\begin{enumerate}
\item $U$ is induced by an open formal subscheme $U_\formal \subset X_\formal$;
\item the schematic closure of $U_0$ in $X_0$ is proper over $S_0$.
\end{enumerate}

% More generally, let $Y \to S$ be a morphism of rigid spaces where $Y$ is quasi-separated and $S$ is quasi-compact and quasi-separated.
% Let $A$ be any subset of~$Y$.
% We say that $A$ is relatively $S$-compact in~$Y$ if there exist quasi-compact open subspaces $U \subset X \subset Y$ such that $A \subset U \Subset_S X$.

A general subset of $X$ is said to be \defterm{relatively $S$-compact} in~$X$ if it is contained in a relatively $S$-compact open subspace of~$X$.
\end{definition}

From the remarks after \cite[Def.~3.6.1]{Lut16}, we note the following properties of this definition.

\begin{remark} \label{relatively-compact-every-model}
If an open subspace $U \subset X$ is relatively $S$-compact, then part~(2) of the definition holds for every formal model $X_\formal \to S_\formal$ such that $U$ is induced by an open formal subscheme of $X_\formal$.
\end{remark}

\begin{remark} \label{relatively-compact-analytic-def}
If $X$, $S$ and~$U$ are all affinoid, then this definition of relative $S$-compactness is equivalent to Kiehl's definition \cite[Def.~2.1]{Kie67}:
$U$ is relatively $S$-compact in~$X$ if there exists a closed immersion $X \to D(1)^n \times S$ which maps $U$ into $D(r)^n \times S$ for some $r<1$.
\end{remark}

\subsection{Fundamental lemmas about relatively compact subspaces}

The following lemmas generalise \cite[Lemma~9.6.2/1]{BGR84} and \cite[Prop.~1.4]{Lam99}, respectively, from affinoid spaces to quasi-compact quasi-separated spaces.

\begin{lemma} \label{relatively-compact-base-change2}
Let $X_1 \to Y$ and $X_2 \to Y$ be two morphisms of quasi-compact, quasi-separated rigid spaces.
Let $U_1$ be an open subspace of $X_1$.
If $U_1 \Subset_Y X_1$, then $U_1 \times_Y X_2 \Subset_{X_2} X_1 \times_Y X_2$.
\end{lemma}

\begin{proof}
Choose formal models $X_{1,\formal} \to Y_\formal$ and $X_{2,\formal} \to Y_\formal$ such that $U_1$ is induced by a formal open subscheme $U_\formal \subset X_{1,\formal}$.
Let $\ov U_{1,0}$ denote the schematic closure of $U_{1,0}$ in $X_{1,0}$.

Let $W_\formal = U_{1,\formal} \times_{Y_\formal} X_{2,\formal}$ and $Z_{\formal} = X_{1,\formal} \times_{Y_\formal} X_{2,\formal}$.
Let $\ov W_0$ denote the schematic closure of $W_0$ in~$Z_0$.

Now $\ov U_{1,0} \times_{Y_0} X_{2,0}$ is a closed subscheme of $Z_0$ through which $W_0 \hookrightarrow Z_0$ factors.
Thus $\ov W_0$ is a closed subscheme of $\ov U_{1,0} \times_{Y_0} Z_0$.

By the hypothesis of the lemma, $\ov U_{1,0} \to Y_0$ is proper.
Hence, by base change, $\ov U_{1,0} \times_{Y_0} X_{2,0} \to X_{2,0}$ is proper.
Composing this with the closed immersion $\ov W_0 \to \ov U_{1,0} \times_{Y_0} Z_0$, we deduce that $\ov W_0 \to X_{2,0}$ is proper.
Thus $W \Subset_{X_2} Z$, as required.
\end{proof}

\begin{lemma} \label{relatively-compact-union}
Let $Y \to S$ be a morphism of quasi-compact, quasi-separated rigid spaces.
Let $U \subset Y$ be a quasi-compact open subspace.
Then the following are equivalent:
\begin{enumerate}[(i)]
\item $U \Subset_S Y$;
\item there exist finitely many quasi-compact open subspaces $V_i \subset Y$ such that $U \subset \bigcup_i V_i$ and $V_i \Subset_S Y$.
\end{enumerate}
\end{lemma}

\begin{proof}
(i) $\Rightarrow$ (ii) is obvious: just take $V_1 = U$.

(ii) $\Rightarrow$ (i): Choose a formal $R$-model $Y_\formal \to S_\formal$ for $Y \to S$ such that $U$ and $V_i$ are induced by open formal subschemes $U_\formal, V_{i,\formal} \subset Y_\formal$.
Let $\ov U_0$, $\ov V_{i,0}$ denote the schematic closures of $U_0$, $V_{i,0}$, respectively, in $Y_0$.
By \cref{relatively-compact-every-model}, each $\ov V_{i,0}$ is proper over~$S_0$.

Let $W_0 = \bigcup_i \ov V_{i,0} \subset Y_0$.
This is a Zariski closed subset of~$Y_0$, and we give it the reduced closed subscheme structure induced from~$Y_0$.
By \cite[Cor.~5.4.6]{EGAII}, each $(\ov V_{i,0})_\red$ is proper over~$S_0$.
Hence, by \cite[Cor.~5.4.5]{EGAII}, $W_0 \to S_0$ is proper.
We have $\ov U_0 \subset W_0$, so there is a closed immersion of schemes $(\ov U_0)_\red \to W_0$.
Hence $(\ov U_0)_\red \to S_0$ is proper.
By \cite[Cor.~5.4.6]{EGAII}, it follows $\ov U_0 \to S_0$ is proper.  Thus $U \Subset_S Y$.
\end{proof}

\subsection{Thickening of relatively compact subspaces} \label{subsec:thickening}

For $0 < r \leq 1$, define the following rigid spaces over~$K$:
\begin{align*}
    A(r) & = \{ x : r \leq \abs{x} \leq r^{-1} \},
\\  A^{\leq}(r) & = \{ x : r \leq \abs{x} \leq 1 \},
\\  A^{\geq}(r) & = \{ x : 1 \leq \abs{x} \leq r^{-1} \}.
\end{align*}

The following lemma, while elementary, is the main additional ingredient required for the proof of \cref{hom-extends-thm} beyond simply relativising the proof of \cite[Prop.~7.3.1]{Lut16}.
Its proof was suggested to the author by Lütkebohmert.

\begin{lemma} \label{relatively-compact-thickening}
Fix $r > 1$.
Let $S$ be a quasi-compact, quasi-separated rigid space.
Let $Y \subset A(r) \times S$ be a quasi-compact open subspace such that $A(1) \times S \Subset_S Y$.
Then $Y$ contains $A(r') \times S$ for some $r' < 1$.
\end{lemma}

\begin{proof}
Note that $S$, $U := A(1) \times S$, $X := A(r) \times S$, $Y$ are quasi-compact, quasi-separated rigid spaces.
Hence we may choose admissible formal $R$-models $S_\formal$, $U_\formal$, $X_\formal$, $Y_\formal$ for $S$, $U$, $X$, $S$ respectively.
We may choose these models such that $U_\formal$ is an open formal subscheme of~$Y_\formal$, and $Y_\formal$ is an open formal subscheme of $X_\formal$.
Since $U \Subset_S Y$, the schematic closure $\ov U_0$ of $U_0$ in $Y_0$ is proper over~$S_0$.

Let $\spmap \colon X \to X_0$ denote the specialisation map.

Let $Z_0 = X_0 \setminus \ov U_0$, which is an open subscheme of~$X_0$.
Let $Z = \spmap^{-1}(Z_0)$, which is an open subset of~$X$.
Let
\[ Z^\leq = \{ (x,s) \in Z : \abs{x} \leq 1 \}, \quad Z^\geq = \{ (x,s) \in Z : \abs{x} \geq 1 \}. \]
Let
\[ r^\leq = \sup \{ \abs{x} : (x,s) \in Z^\leq \}, \quad r^\geq = \sup \{ \abs{x}^{-1} : (x,s) \in Z^\geq \}. \]
(If $Z^\leq$ or $Z^\geq$ empty, set $r^\leq=r$ or $r^\geq=r$ respectively.)

We claim that $r^\leq < 1$.
If $Z^\leq$ is empty, then this is immediate.
Otherwise, assume that $Z^\leq$ is non-empty.
Since $Z_0$ is quasi-compact, so is~$Z$.
Since $X$ is quasi-separated, it follows that $Z^\leq$ is quasi-compact.
Hence, by the maximum principle, the absolute value of the first co-ordinate $\abs{x}$ attains its supremum~$r^\leq$ on~$Z^\leq$.
But $Z$ is disjoint from $U = A(1) \times S$, so $\abs{x} \neq 1$ on $Z^\leq$.
Hence $r^\leq < 1$.

Similarly, we obtain $r^\geq < 1$.

Choose $r'$ such that $\max\{ r^\leq, r^\geq \} < r' < 1$.
We claim that $A^\leq(r') \times S \subset Y$.
Indeed, if $(x,s) \in A^\leq(r') \times S$, then $r^\leq < r' \leq \abs{x}$.
Hence by the definition of~$r^\leq$, $(x,s) \not \in Z^\leq$.
Since $\abs{x} \leq 1$, it follows that $(x,s) \not \in Z$.
Hence
\[ \spmap((x,s)) \in X_0 \setminus Z_0 = \ov U_0 \subset Y_0. \]
Since $Y_\formal$ is an open formal subscheme of~$X_\formal$, it follows that $(x,s) \in Y$, as required.

A similar argument shows that $A^\geq(r') \times S \subset Y$.
Thus $A(r') \times S \subset Y$, completing the proof of the lemma.
\end{proof}

\subsection{Approximation theorem}

The approximation theorem asserts that, under suitable conditions, a morphism from an open subspace of a rigid space~$Z$ to a rigid space~$X$, with relatively compact image, may be approximated by a morphism from a larger open subspace of~$Z$ to~$X$.
A slightly weaker version of the theorem is found at \cite[Thm.~3.6.7]{Lut16}.

In the following theorem statement, $S$ denotes Raynaud's rigid generic fibre of $S_\formal$, etc.
Fix $\pi \in K^\times$ with $\abs{\pi} < 1$.

\begin{theorem} \label{approximation:formal} \cite[Thm.~7.4]{Lut95}
Let $S_\formal$ be an affine admissible formal $R$-scheme.
Let $X_\formal$ be a quasi-compact separated admissible formal $S_\formal$-scheme such that the induced rigid morphism $X \to S$ is smooth.

Let $Z_\formal$ be a quasi-compact admissible formal $S_\formal$-scheme and let $U_\formal$ be an open formal subscheme of $Z_\formal$.
Suppose that $Z$ is affinoid and $U$ is a Weierstrass domain in~$Z$.

Let $\phi_\formal \colon U_\formal \to X_\formal$ be a morphism of formal schemes over $S_\formal$, such that $\phi(U) \Subset_S X$.

Let $\lambda \in \ZZ_{\geq 1}$.
Then there exists an admissible formal blowing-up $\psi_\formal \colon Z'_\formal \to Z_\formal$, which is finite over~$U_\formal$, and an open formal subscheme $U'_\formal \subset Z'_\formal$, such that:
\begin{enumerate}[(a)]
\item the schematic closure of $(Z'_\formal \times_{Z_\formal} U_\formal)_0$ (in $Z'_0$) is contained in $U'_0$;
\item the schematic closure $\ov U'_0$ (in $Z'_0$) is proper over $\ov U_0$ (the closure of~$U_0$ in~$Z_0$);
\item there exists an morphism $\phi_\formal' \colon U'_\formal \to X_\formal$ such that 
\[ \phi_\formal'|_{Z'_\formal \times_{Z_\formal} U_\formal} \equiv \phi_\formal \circ \psi_\formal|_{Z'_\formal \times_{Z_\formal} U_\formal} \bmod \pi^\lambda. \]
\end{enumerate}
\end{theorem}

The following additional information can be obtained from \cite[Thm.~7.4]{Lut95} and its proof.
The analogue of~(d) is stated in \cite[Satz~5.2.1]{Mar08} and \cite[Thm.~6.3]{Lut09}, which give another version of the approximation theorem under the additional hypothesis $U \Subset Z$ (this additional hypothesis is not satisfied in our application).

\begin{proposition} \label{approximation:extra}
In the setting of \cref{approximation:formal}, the following additional properties hold:
\begin{enumerate}[(a)]
\setcounter{enumi}{3}
\item $\phi_\formal'$ is a morphism of formal $S_\formal$-schemes;
\item $U \Subset_Z U'$.
\end{enumerate}
\end{proposition}

\begin{proof}
Claim~(d) may be seen by examining the proof of \cite[Thm.~7.4]{Lut95}.
    
For claim~(e), note that from \cref{approximation:formal}(a), $U_\formal^* := Z'_\formal \times_{Z_\formal} U_\formal$ is an open formal subscheme of~$U'_\formal$.
Also from (a), the schematic closure $\ov U^*_0$ of $U^*_0$ in $U'_0$ is equal to the schematic closure of $U^*_0$ in $\ov U'_0$.
Hence by~(b), $\ov U^*_0$ is proper over~$\ov U_0$.
Since $\ov U_0 \to Z_0$ is a closed immersion, it follows that $\ov U^*_0$ is proper over~$Z_0$ and, hence, $U \Subset_Z U'$.
\end{proof}

\section{Group laws on relative rigid balls} \label{sec:group-laws-balls}

In this section, we consider rigid $S$-groups whose underlying rigid space is either an $S$-ball (either closed or open).  The group law for such a group is a $g$-dimensional formal group law with coefficients in $\cO^\circ(S)$.  For these groups, we prove several elementary lemmas describing how the group law is approximated by addition in the ball.

Fix $\pi \in K^\times$ with $\abs{\pi} < 1$.
Let $\bar R$ denote the ring of integers of~$\bar K$.

Let $S$ be a reduced affinoid space.
Let $U$ be a rigid $S$-group such that the underlying rigid $S$-space is isomorphic to either the relative closed unit ball $D^g \times S$ or the relative open unit ball $D_+^g \times S$.
Fix an isomorphism of rigid $S$-spaces $\zeta \colon U \to D^g \times S$ or $U \to D_+^g \times S$, such that $\zeta$ composed with the identity section $S \to U$ is equal to the zero section $S \to D^g \times S$ (such an isomorphism always exists, as we may take an arbitrary isomorphism $U \to D^g \times S$ or $U \to D_+^g \times S$ and translate it).
For $i=1, \dotsc, g$, we write $p_i \colon D^g \times S \to D$ or $D_+^g \times S \to D_+$ for the projection onto the $i$-th factor, and we write $\zeta_i = p_i \circ \zeta \in \cO^\circ(U)$.

Via $\zeta$, the $S$-group law $U \mathbin{\times_S} U \to U$ can be described by a $g$-tuple of power series
\[ F = (F_1, \dotsc, F_g) \in \powerseries{\cO(S)}{X_1, \dotsc, X_g, Y_1, \dotsc, Y_g}^g. \]
Likewise the inverse $U \to U$ can be described by a $g$-tuple of power series
\[ i \in \powerseries{\cO(S)}{X_1, \dotsc, X_g}^g. \]

The following lemma implies that $F$ and~$i$ have coefficients in $\cO^\circ(S)$.
This is obvious in the case $U \cong D^g \times S$; indeed, in this case, the components of~$F$ are in $\cO^\circ(D^g \times D^g \times S) = \tatealgebra{\cO^\circ(S)}{X_1, \dotsc, X_g, Y_1, \dotsc, Y_g}$, and similarly for~$i$.
The lemma shows that it is also true for $U \cong D_+^g \times S$.

\begin{lemma}
Every holomorphic function $f \colon D_+^m \times S \to D$ is represented by a power series in $\powerseries{\cO^\circ(S)}{X_1, \dotsc, X_m}$.
\end{lemma}

\begin{proof}
The restriction of $f$ to $D(\epsilon)^m \times S$ is represented by an element of 
\[ \cO(D(\epsilon)^m \times S) = \tatealgebra{\cO(S)}{\pi^{-1}X_1, \dotsc, \pi^{-1}X_m} \subset \powerseries{\cO(S)}{X_1, \dotsc, X_m}. \]
Thus, we may write
\[ f = \sum_{j \in \ZZ_{\geq0}^m} f_j \underline X^j, \]
where $f_j \in \cO(S)$.

For every $n \in \ZZ_{>0}$, $f$ maps $D(\epsilon^{1/n})^m \times S$ into~$D$, so we have
\[ 1 \geq \abs{f}_{D(\epsilon^{1/n})^m \times S} = \max \bigl\{ \epsilon^{\abs{j}/n}\abs{f_j}_S : j \in \ZZ_{\geq0}^m \bigr\}. \]
Hence, for each~$j \in \ZZ_{\geq0}^n$ and each $n \in \ZZ_{>0}$, we have $\abs{f_j}_S \leq \epsilon^{-\abs{j}/n}$.
If we fix $j$ and let $n \to \infty$, then $\epsilon^{-\abs{j}/n} \to 1$.
Hence $\abs{f_j}_S \leq 1$, that is, $f_j \in \cO^\circ(S)$, for all $j \in \ZZ_{\geq0}^n$.
\end{proof}

Since $U$ is an $S$-group and $\zeta$ maps the identity section of~$U$ to the zero section, $F$ is a $g$-dimensional formal group law over $\cO^\circ(S)$.
In other words,
\begin{enumerate}
\item $F(\underline X, \underline 0) = \underline X$ and $F(\underline 0, \underline Y) = \underline Y$ (identity section of~$U$ maps to $\underline 0$);
\item $F(\underline X, F(\underline Y, \underline Z)) = F(F(\underline X, \underline Y), \underline Z)$ (associativity);
\item $F(\underline X, i(\underline X)) = F(i(\underline X), \underline X) = \underline 0$ (inverse).
\end{enumerate}

\begin{notation} \label{notation:congruence}
Let $\lambda \in \ZZ_{>1}$ and let $T$ be a reduced affinoid space with a morphism $T \to S$.

If $x,y$ are $S$-morphisms $T \to D^g \times S$ or $T \to D_+^g \times S$, write $x \equiv y \bmod \pi^\lambda$ to mean that $\abs{p_i \circ x -  p_i \circ y}_T \leq \epsilon^\lambda$ for all~$i=1,\dotsc,g$.
Equivalently, $x \equiv y \bmod \pi^\lambda$ means that $p_i(x(t)) \equiv p_i(y(t)) \bmod \pi^\lambda$ for all $t \in T(\bar K)$.

If $u,v$ are $S$-morphisms $T \to U$, write $u \equiv v \bmod \pi^\lambda$ to mean that $\zeta \circ u \equiv \zeta \circ v \bmod \pi^\lambda$ in the sense defined above.
\end{notation}

\begin{notation} \label{notation:Ur}
For $r \in \abs{\ov K^\times}$ with $0<r<1$, write
\[ U[r] = \{ u \in U : \abs{\zeta_i(u)} \leq r \text{ for all } i=1, \dotsc, g \}. \]
\end{notation}

Note that the notations $u \equiv v$ and $U[r]$ depend implicitly on the isomorphism~$\zeta$.

\begin{lemma} \label{formal-group-congruence-piU}
Let $\lambda \in \ZZ_{>0}$.
Let $T$ be a reduced affinoid $S$-space and let $u,v$ be $S$-morphisms $T \to U$.
We have $u \equiv v \bmod \pi^\lambda$ if and only if the image of $u \cdot v^{-1}$ is contained in $U[\epsilon^\lambda]$.
\end{lemma}

\begin{proof}
Let $w = u \cdot v^{-1} \colon T \to U$.
Consider $t \in T(\bar K)$ and write
\[ \zeta(u(t)) = (\underline x, s), \quad \zeta(v(t)) = (\underline y, s), \quad \zeta(w(t)) = (\underline z, s), \]
where $\underline x, \underline y, \underline z \in \bar R^g$ and $s \in S(\bar K)$.

For the forwards implication, if $u \equiv v \bmod \pi^\lambda$, then we have $\underline x \equiv \underline y \bmod \pi^\lambda$ in $\bar R^g$.
Since the power series $F_1, \dotsc, F_g, i_1, \dotsc, i_g$ have coefficients in $\cO^\circ(S)$, this implies that
\[ F(\underline x, i(\underline y, s), s) \equiv F(\underline x, i(\underline x, s), s) \bmod \pi^\lambda. \]
By the definition of~$F$ and~$i$, and the axiom for the inverse of a formal group law, we can now calculate
\[ \underline z = F(\underline x, i(\underline y, s), s)  \equiv F(\underline x, i(\underline x, s), s) = \underline 0 \bmod \pi^\lambda. \]
Since the above holds for all $t \in T(\bar K)$, we obtain $\abs{\zeta_i \circ w}_X \leq \epsilon^\lambda$ for each~$i=1, \dotsc, g$.
This means that $u \cdot v^{-1}$ factors through $U[\epsilon^\lambda]$, as required.

For the converse implication, if the image of $u \cdot v^{-1}$ is contained in $U[\epsilon^\lambda]$, then we have $\abs{\zeta_i(w(t))} \leq \epsilon^\lambda$ for each~$i$, so $\underline z \equiv \underline 0 \bmod \pi^\lambda$.
Hence
\[ \underline x = F(\underline z, \underline y) \equiv F(\underline 0, \underline y) = \underline y \bmod \pi^\lambda, \]
where the first equality follows from $u = w \cdot v$ and the last equality from the identity axiom for a formal group law.
Again, since this holds for all $t \in T(\bar K)$, we obtain that $u \equiv v \bmod \pi^\lambda$.
\end{proof}

\begin{corollary} \label{formal-group-congruence-subgroup}
$U[\epsilon^\lambda]$ is an $S$-subgroup of~$U$.
\end{corollary}

\begin{proof}
Let $p_1,p_2$ denote the morphisms $U[\epsilon^\lambda] \times_S U[\epsilon^\lambda] \to U$ obtained by composing projections onto the factors with the inclusion $U[\epsilon^\lambda] \to U$.
We have $p_1 \equiv \underline 0 \equiv p_2 \bmod \pi^\lambda$.
Applying \cref{formal-group-congruence-piU} with $u = p_1$, $v=p_2$, we deduce that $p_1 \cdot p_2^{-1}$ maps $U[\epsilon^\lambda] \times_S U[\epsilon^\lambda]$ into $U[\epsilon^\lambda]$, as required.
\end{proof}

\begin{lemma} \label{formal-group-F2-congruence}
Let $T$ be a reduced affinoid $S$-space and let $u,v$ be $S$-morphisms $T \to U[\epsilon^\lambda]$.
Then
\[ \zeta(u \cdot v) \equiv \zeta(u) + \zeta(v) \bmod \pi^{2\lambda}. \]
\end{lemma}

\begin{proof}
Consider $t \in T(\bar K)$ and write
\[ \zeta(u(t)) = (\underline x, s), \quad \zeta(v(t)) = (\underline y, s), \]
where $\underline x, \underline y \in \bar R^g$ and $s \in S(\bar K)$.

From condition~(1) for a formal group law, we obtain that
\[ F(\underline X, \underline Y) = \underline X + \underline Y + \text{higher order terms}, \]
as an equation in $\tatealgebra{\cO^\circ(S)}{\underline X, \underline Y}$, where each higher order term is divisible by at least one $X_j$ variable and by at least one $Y_k$ variable.
Hence, since $\abs{x_j}, \abs{y_k} \leq \epsilon^\lambda$, the higher order terms, evaluated at $(\underline x, \underline y, s)$, all have absolute value at most $\epsilon^{2\lambda}$.
Therefore, $F(\underline x, \underline y, s) \equiv \underline x + \underline y \bmod \pi^{2\lambda}$ for all such $\underline x, \underline y, s$.

Thus
\[ \zeta(u(t) \cdot v(t)) \equiv \zeta(v(t)) + \zeta(v(t)) \bmod \pi^{2\lambda} \]
for all $t \in T(\bar K)$, proving the lemma.
\end{proof}

\section{Main proof} \label{sec:main-proof}

In this section, we prove the main theorem of the paper, \cref{hom-extends-thm}.

For $0 < \epsilon \leq 1$, define $\GG_m(\epsilon)$ to be the following rigid space over~$K$:
\[ \GG_m(\epsilon) = \{ x : \epsilon \leq \abs{x} \leq \epsilon^{-1} \}. \]
Note that $\GG_m(\epsilon) = A(\epsilon)$ in the notation of section~\ref{subsec:thickening}, and $\ov\GG_m = \GG_m(1)$.
However, $\GG_m(\epsilon)$ is not a subgroup of~$\GG_m$ when $\epsilon < 1$.

\subsubsection*{Outline of the proof}

First, note that the uniqueness of $\phi$ holds by the identity principle for rigid morphisms.
As a consequence of the uniqueness, it suffices to construct $\phi$ locally on an admissible covering of~$S$ and glue, so we can and do assume that $S$ is affinoid and that $S_\formal$ is affine.
Likewise, if we construct $\phi$ over a finite extension of~$K$, it will descend to~$K$, so we may freely replace $K$ by finite extensions (and base-change all rigid spaces accordingly) in the course of the proof.

We use an Approximation Theorem of Lütkebohmert (\cref{approximation:formal}) to show that there exists a morphism $\alpha \colon \GG_m(\epsilon^2) \times S \to G$ for some $\epsilon < 1$, which approximates $\bar\phi$ in the sense that $\bar\phi \equiv \alpha|_{\ov\GG_m \times S} \bmod \pi$ relative to the smooth formal model $H_\formal$ given by \cref{hom-extends-thm}(i).
In other words, if we define $\bar u := \alpha|_{\ov\GG_m \times S} \cdot \bar\phi^{-1}$, then $\bar u$ factors through the tube~$T$ of the zero section of $H_\formal$.
Berthelot's weak fibration theorem establishes that, after passing to a suitable admissible covering of~$S$, $T \cong D_g^+ \times S$ as rigid $S$-spaces.

By the maximum principle, the image of $\bar u$ is contained in a closed $S$-ball inside~$T$.  Indeed, by choosing a slightly larger closed $S$-ball~$U \subset T$ and increasing $\epsilon$, we may arrange that the image of $\bar u$ lies inside the closed $S$-ball $U[\epsilon] \Subset_S U$ (defined as in \cref{notation:Ur}).
The calculations from section~\ref{sec:group-laws-balls} show that $U$ and $U[\epsilon]$ are $S$-subgroups of~$T$, hence~$G$.

We now measure the failure of $\alpha \colon \GG_m(\epsilon^2) \times S \to G$ to be an $S$-group homomorphism by introducing the morphism
\[ w_\alpha(x_1, x_2, s) = \alpha(x_1^{-1} x_2^{-1}, s) \alpha(x_1, s) \alpha(x_2, s), \]
which is well-defined as a morphism $\GG_m(\epsilon) \times \GG_m(\epsilon) \times S \to G$.
After increasing $\epsilon$, we arrange that $w_\alpha$ factors through~$U$.
Since $\bar u$ and $w_\alpha$ both factor through the $S$-ball~$U$, we may represent them by power series and use these for calculations.

The heart of the proof is an inductive argument that $\bar u \colon \ov\GG_m \times S \to U$ can be approximated arbitrarily closely by morphisms~$u_\lambda$ which extend to $\GG_m(\epsilon^2) \times S \to U$.
The base case is simply that $\bar u$ factors through~$U[\epsilon]$; in other words, it is approximated by~$0$ to order~$\epsilon$.
For the inductive step, we introduce a morphism $w_\lambda$ analogous to $w_\alpha$, measuring the failure of $\alpha \cdot u_\lambda^{-1}$ to be an $S$-group homomorphism.
Calculations with the Laurent series of~$w_\lambda$ allow us to construct a better approximation $u_{\lambda+1}$ to $\bar u$ which still extends to $\GG_m(\epsilon^2) \times S$.

The facts that $u_\lambda \to \bar u$ in $\cO^\circ(\ov\GG_m \times S)^g$, while the morphisms $u_\lambda$ are defined on $\GG_m(\epsilon^2) \times S$, imply that the power series for $\bar u$ converge on the intermediate annulus $\GG_m(\epsilon) \times S$.
In other words, $\bar u$ extends to a morphism $u \colon \GG_m(\epsilon) \times S \to U$.
It follows that $\ov\phi \colon \ov\GG_m \times S \to G$ extends to a morphism $\psi = \alpha \cdot u^{-1} \colon \GG_m(\epsilon) \times S \to G$, which satisfies the condition to be a group homomorphism whenever it makes sense.
Since $\GG_m$ has an admissible covering by translates of the annulus $\GG_m(\epsilon)$, we can use $\psi$ to define the desired homomorphism $\phi \colon \GG_m \times S \to G$.

\subsection{Approximation}

The proof of \cref{hom-extends-thm} begins with the application of the Approximation Theorem to $\bar\phi$.

Let $\ov\GG_{m,\formal} = \Spf(\tatealgebra{R}{\xi, \xi^{-1}})$ be the standard affine formal model for~$\ov\GG_m$.

\begin{lemma} \label{alpha-approximation}
Let $W = \ov\GG_m \times S$.
In the situation of \cref{hom-extends-thm},
for suitable formal $R$-models $W_\formal$ of~$W$ and $\bar\phi_\formal \colon W_\formal \to H_\formal$ of $\bar\phi \colon W \to H$, there exists an $S_\formal$-morphism $\bar\alpha_\formal \colon W_\formal \to H_\formal$ such that
\begin{enumerate}[(a)]
\item $\bar\alpha \colon W \to H$ (the rigid morphism induced by~$\bar\alpha_\formal$) extends to a rigid morphism $\alpha \colon \GG_m(\epsilon^2) \times S \to G$, for some $\epsilon$ satisfying $0 < \epsilon < 1$;
\item $\bar\alpha_0 = \bar\phi_0 \colon W_0 \to H_0$.
\end{enumerate}
\end{lemma}

\begin{proof}
Let $Z = \GG_m(\abs{\pi}) \times S$, so that $W$ is an open subspace of~$Z$.
Let $X' = X \cap H$, where $X$ is the subspace of~$G$ which appears in \cref{hom-extends-thm}(ii).
By hypotheses (i) and~(ii) of \cref{hom-extends-thm}, $\bar\phi \colon W \to G$ factors through~$X'$.

We construct formal models for $H$, $W$, $X$, $X'$, $Z$ and $\bar\phi$ as follows:
\begin{enumerate}[(i)]
\item By \cite[Cor.~5.9 and Lemma~5.6]{BL93II}, there exist formal models $W_\formal$, $X_\formal$ and $Z_\formal$ of $W$, $X$ and~$Z$ respectively, together with morphisms $X_\formal \to S_\formal$ and $Z_\formal \to S_\formal$ inducing the structure morphisms $W \to S$, $X \to S$ and $Z \to S$.
\item By \cite[Lemma~5.6]{BL93II}, after replacing $X'_\formal$ by an admissible formal blowing-up, there exists a morphism $X'_\formal \to H_\formal$ inducing the inclusion $X' \to H$.
\item By \cite[Cor.~5.4(b)]{BL93II}, after replacing $X_\formal$ and $X'_\formal$ by admissible formal blowing-ups, there exists an open immersion $X'_\formal \to X_\formal$ inducing the inclusion $X' \to X$.
\item By \cite[Lemma~5.6]{BL93II}, after replacing $W_\formal$ by an admissible formal blowing-up,
there exists a morphism $\bar\phi_\formal' \colon W_\formal \to X'_\formal$ inducing $\bar\phi \colon W \to X'$.
\item By \cite[Cor.~5.4(b)]{BL93II}, after replacing $W_\formal$ and $Z_\formal$ by admissible formal blowing-ups, there exists an open immersion $W_\formal \to Z_\formal$ inducing the inclusion $W \to Z$.
\end{enumerate}

By \cite[Prop.~4.7]{BL93I}, $X_\formal \to S_\formal$ is separated.
By the hypotheses of \cref{hom-extends-thm}, $X \to S$ is smooth.
By construction, $W_\formal$ is an open formal subscheme of $Z_\formal$, $Z$ is affinoid and $W$ is a Weierstrass domain in~$Z$.
By \cite[p.~307, (b)]{BL93I}, $\bar\phi_\formal' \colon W_\formal \to X_\formal$ is a morphism of $S_\formal$-schemes.
According to \cref{hom-extends-thm}(ii), $\bar\phi(W) \Subset_S X$.
Hence the conditions of \cref{approximation:formal} are satisfied by $S_\formal$, $X_\formal$, $Z_\formal$, $U_\formal := W_\formal$ and $\phi_\formal := \bar\phi'_\formal$.

Applying \cref{approximation:formal} and \cref{approximation:extra}, we replace $Z_\formal$ by the resulting admissible formal blowing-up $Z'_\formal$ (and pull back $W_\formal$ so that it remains an open formal subscheme of $Z_\formal$).  We obtain an open formal subscheme $W'_\formal \subset Z_\formal$ and an  $S_\formal$-morphism $\alpha'_\formal \colon W'_\formal \to X_\formal$ with the following properties:
\begin{enumerate}
\item $W \Subset_Z W'$;
\item $\alpha'_\formal|_{W_\formal} \equiv \bar\phi'_\formal \bmod \pi$; in other words, $\alpha'_0|_{W_0} = \bar\phi'_0$ as morphisms of $k$-schemes $W_0 \to X_0$.
\end{enumerate}

Thanks to~(2), $\alpha'_0|_{W_0}$ factors through the open subscheme $X'_0 \subset X_0$.
Hence, $\alpha'_\formal|_{W_\formal}$ factors through the open formal subscheme $X'_\formal \subset X_\formal$.
Let 
\[ \bar\phi_\formal, \bar\alpha_\formal \colon W_\formal \to H_\formal \]
denote the compositions of $\bar\phi'_\formal$ and $\alpha_\formal|_{W_\formal}$ respectively with the morphism $X'_\formal \to H_\formal$.
Then, again as a consequence of~(2), $\bar\alpha_0 = \bar\phi_0$ as morphisms of $k$-schemes $W_0 \to H_0$.
Thus conclusion~(b) of the lemma holds.

By construction, $W \Subset_S Z$.
Hence, by~(1) and by \cite[Remark~3.6.2(b)]{Lut16}, $W \Subset_S W'$.
So, by \cref{relatively-compact-thickening}, $W'$ contains $\GG_m(r') \times S$ for some $r'<1$.
Furthermore, the rigid morphism $\alpha' \colon W' \to X$ induced by $\alpha'_\formal$ extends $\bar\alpha \colon W \to X'$.
Thus conclusion~(a) of the lemma holds with $\epsilon = \sqrt{r'}$.
\end{proof}

\begin{lemma} \label{baru-through-tube}
In the situation of \cref{hom-extends-thm}, there exists an admissible covering $\{ S_i \}$ of~$S$, a real number $\epsilon$ satisfying $0 < \epsilon < 1$, a rigid $S$-morphism $\alpha \colon \GG_m(\epsilon^2) \times S \to G$, and an open $S$-subgroup $T \subset G$ such that:
\begin{enumerate}[(a)]
\item $\bar u := \alpha|_{\ov\GG_m \times S} \cdot \bar\phi^{-1}$ factors through~$T$;
\item for each index~$i$, $T \times_S S_i$ is isomorphic as a rigid $S_i$-space to $D_g^+ \times S_i$.
\end{enumerate}
\end{lemma}

\begin{proof}
Let $T \subset H$ denote the tube over the image of the zero section $S_0 \to H_0$ (with respect to the formal model $H_\formal$).
According to \cref{hom-extends-thm}(i), $H_\formal$ is an $S_\formal$-group, so $T$ is an $S$-subgroup of~$H$.

Apply the weak fibration theorem \cite[Thm.~1.3.2]{Berthelot} with $X=S_0$, $P=S_\formal$, $P'=H_\formal$, $i=$ the closed immersion $S_0 \to S_\formal$, $i'=$ the zero section $S_0 \to H_0$ composed with the closed immersion $H_0 \to H_\formal$, $u=$ the structure morphism $H_\formal \to S_\formal$ (which is smooth by \cref{hom-extends-thm}(i)).
Following the notation of \cite{Berthelot}, we have ${]X[_P} = S$ and ${]X[_{P'}} = T$.
Hence, by \cite[Thm~1.3.2]{Berthelot}, there exists an admissible covering $\{ S_i \}$ of~$S$ which satisfies conclusion~(b) of the proposition.

Let $\bar\alpha_\formal \colon \ov\GG_{m,\formal} \times S_\formal \to H_\formal$ and $\alpha \colon \GG_m(\epsilon^2) \times S \to G$ be the morphisms given by \cref{alpha-approximation}.
Let
\[ \bar u_\formal = \bar\alpha_\formal \cdot \bar\phi_\formal^{-1} \colon \ov\GG_{m,\formal} \times S_\formal \to H_\formal. \]
By \cref{alpha-approximation}(b), $\bar u_0$ is the zero homomorphism of $S_0$-group schemes $\GG_m \times S_0 \to H_0$.
Hence, the associated rigid morphism $\bar u$ factors through~$T$, proving conclusion~(a) of the proposition.
\end{proof}

\subsection{Properties of \texorpdfstring{$\bar u$}{u-bar}}

From now on, we place ourselves in the situation of \cref{hom-extends-thm}, with $\alpha$, $\bar u$ and $T$ given by \cref{baru-through-tube}.
Using again the fact that it suffices to prove \cref{hom-extends-thm} locally on an admissible covering of~$S$, we may replace $S$ by~$S_i$ and assume that $T \cong D_+^g \times S$ as a rigid $S$-space.
Refining the admissible covering, we may still assume that $S$ is affinoid.

In the following steps, we will increase the parameter $\epsilon$ (while retaining $0 < \epsilon < 1$).
This does not affect the fact from \cref{baru-through-tube}(a) that $\alpha$ is defined on $\GG_m(\epsilon^2)$, because increasing $\epsilon$ makes $\GG_m(\epsilon^2)$ smaller.

We now construct a closed $S$-ball $U \subset G$ such that all the morphisms we need factor through~$U$, and hence can be represented by restricted power series.

\begin{lemma} \label{baru-piU}
After replacing $K$ by a finite extension and increasing $\epsilon$, there exists an open $S$-subgroup $U \subset H$, which is isomorphic as a rigid $S$-space to $D^g \times S$ via an isomorphism $\zeta \colon U \to D^g \times S$, such that $\bar u \colon \ov\GG_m \times S \to H$ factors through
$U[\epsilon]$ (using \cref{notation:Ur}).
\end{lemma}

\begin{proof}
Choose an isomorphism of rigid $S$-spaces $\eta \colon T \to D_+^g \times S$.
By the maximum principle, since $\ov\GG_m \times S$ is affinoid, the image of
\[ \eta \circ \bar u \colon \ov\GG_m \times S \to D_+^g \times S \]
is contained in $D(r)^g \times S$ for some $r<1$.
Replace $K$ by a finite extension and increase $\epsilon$ such that there exists an element $\pi' \in K$ satisfying $\abs{\pi'} = \epsilon$ and $r \leq \epsilon^2 < 1$.

Let $U = T[\epsilon]$.
By \cref{formal-group-congruence-subgroup}, $U$ is an open $S$-subgroup of~$T$ and hence of~$H$.
Since $r \leq \epsilon^2$, $\bar u$ factors through $T[\epsilon^2] = U[\epsilon]$.

Finally, there is an isomorphism $\zeta \colon U \to D^g \times S$ defined by the following composition:
\[ U  \mathrel{\overset{\eta|_U}{\longrightarrow}}  D(\epsilon)^g \times S  \mathrel{\overset{\sim}{\longrightarrow}}  D^g \times S, \]
where the second arrow divides each coordinate by~$\pi'$.
\end{proof}

Note that \cref{baru-piU} will remain true when we further increase $\epsilon$, because increasing $\epsilon < 1$ makes $U[\epsilon]$ larger.

Define a morphism of rigid $S$-spaces $w_\alpha \colon \GG_m(\epsilon) \times \GG_m(\epsilon) \times S \to G$ by
\[ w_\alpha(x_1, x_2, s) = \alpha(x_1^{-1} x_2^{-1}, s) \alpha(x_1, s) \alpha(x_2, s). \]
This is well-defined since $x_1, x_2 \in \GG_m(\epsilon)$ implies that $x_1^{-1}x_2^{-1} \in \GG_m(\epsilon^2)$, and the domain of~$\alpha$ is $\GG_m(\epsilon^2) \times S$.

\begin{lemma} \label{w-alpha-U}
After increasing $\epsilon$, the morphism
$w_\alpha \colon \GG_m(\epsilon) \times \GG_m(\epsilon) \times S \to G$ factors through $U$.
\end{lemma}

\begin{proof}
Let $\bar w_\alpha$ denote the restriction of $w_\alpha$ to $\ov\GG_m \times \ov\GG_m \times S$.
Since $\bar\phi \colon \ov\GG_m \times S \to G$ is a homomorphism of commutative $S$-groups, and using the definition of~$\bar u$ in \cref{baru-through-tube}(a), we have
\begin{equation*} % \label{eqn:barw-baru}
\bar w_\alpha(x_1, x_2, s) = \bar u(x_1^{-1} x_2^{-1}, s) \bar u(x_1, s) \bar u(x_2, s) \text{ on } \ov\GG_m \times \ov\GG_m \times S.
\end{equation*}
Using \cref{baru-piU}, and since $U[\epsilon]$ is an $S$-subgroup of~$G$, it follows that $\bar w_\alpha$ factors through $U[\epsilon]$.

Let $X_2 = \GG_m(\epsilon) \times \GG_m(\epsilon) \times S$.
Since $X_2$ is quasi-compact, there exists a quasi-compact open subspace $Y \subset G$ such that $w_\alpha \colon X_2 \to G$ factors through~$Y$.
After enlarging the quasi-compact space~$Y \subset G$, we may also assume that $U \subset Y$.

By construction, $U[\epsilon] \Subset_S U$ and so $U[\epsilon] \Subset_Y U$.
Applying \cref{relatively-compact-base-change2} to $Y$ and $X_2$ defined in the previous paragraph (with $X_2 \to Y$ given by $w_\alpha$), $X_1=U$ (with $X_1 \to Y$ being the inclusion), and $U_1=U[\epsilon]$, we obtain
\[ w_\alpha^{-1}(U[\epsilon]) = U[\epsilon] \times_Y X_2 \Subset_{X_2} U \times_Y X_2 = w_\alpha^{-1}(U). \]

Since $\bar w_\alpha$ factors through $U[\epsilon]$, we have $\ov\GG_m \times \ov \GG_m \times S \subset w_\alpha^{-1}(U[\epsilon])$. 
Thus it follows that
\[ \ov\GG_m \times \ov \GG_m \times S \Subset_{X_2} w_\alpha^{-1}(U) \subset X_2. \]
By the definition of~$X$, we also have $\ov\GG_m \times \ov\GG_m \times S \Subset_S X_2$.
Hence, by \cite[Rmk.~3.6.2(b)]{Lut16}, we obtain $\ov\GG_m \times \ov\GG_m \times S \Subset_S w_\alpha^{-1}(U)$.

By \cref{relatively-compact-thickening}, it follows that $w_\alpha^{-1}(U)$ contains $\GG_m(\epsilon') \times \GG_m(\epsilon') \times S$ for some $\epsilon'$ such that $\epsilon \leq \epsilon' < 1$.
Replacing $\epsilon$ by $\epsilon'$, the lemma holds.
\end{proof}

After replacing $K$ by a finite extension and increasing $\epsilon$, we may assume that there is a uniformiser $\pi \in K^\times$ with $\abs{\pi} = \epsilon$.

In what follows, we use the notation $\equiv$ as in \cref{notation:congruence}, with respect to $\zeta \colon U \to D^g \times S$.
For $1 \leq i \leq g$, let $\zeta_i \colon U \to D$ denote the composition of~$\zeta$ with the projection onto the $i$-th copy of $D$.

\subsection{Induction}

The core of the proof of \cref{hom-extends-thm} is an induction argument, showing that $\bar u$ can be approximated arbitrarily well by morphisms which extend to $\GG_m(\epsilon^2) \times S$.
We now carry out this induction.
In order to make the structure of the inductive argument clearly visible, we have postponed the calculation at its centre to \cref{coeffs-congruence-lemma} below.

\pagebreak % COSMETIC

\begin{lemma} \label{induction}
For each $\lambda \in \ZZ_{\geq 1}$, there exists an $S$-morphism $u_\lambda \colon \GG_m(\epsilon^2) \times S \to U$ satisfying
\begin{equation} \label{eqn:ui-congruence}
\bar u \equiv u_\lambda|_{\ov\GG_m \times S} \bmod \pi^\lambda.
\end{equation}
\end{lemma}

\begin{proof}
The proof is by induction on~$\lambda$.

For the base case, we take $u_1$ to be the composition of the projection $\GG_m(\epsilon^2) \times S \to S$ with the zero section $S \to U$.
By \cref{baru-piU}, we have $\bar u \equiv u_1 \bmod \pi$.
Thus \eqref{eqn:ui-congruence} holds for $\lambda=1$.

\smallskip

For the inductive step, assume that \eqref{eqn:ui-congruence} holds for a certain $\lambda \geq 1$.
We construct $u_{\lambda+1}$ such that \eqref{eqn:ui-congruence} holds for $\lambda+1$.

Define
\begin{align*}
\bar v_\lambda & = \bar u \cdot (u_\lambda|_{\ov\GG_m \times S})^{-1} \colon \ov\GG_m \times S \to G,
\end{align*}
where $\cdot$ and ${}^{-1}$ refer to the $S$-group law on~$G$.
By the inductive hypothesis~\eqref{eqn:ui-congruence} and \cref{formal-group-congruence-piU}, $\bar v_\lambda$ factors through $U[\epsilon^\lambda]$.

Define $w_\lambda \colon  \GG_m(\epsilon) \times \GG_m(\epsilon) \times S \to G$ by
\[ w_\lambda(x_1, x_2, s) = w_\alpha(x_1, x_2, s) \cdot u_\lambda(x_1^{-1}x_2^{-1}, s)^{-1} \cdot u_\lambda(x_1, s)^{-1} \cdot u_\lambda(x_2, s)^{-1}. \]
Since $w_\alpha$ and $u_\lambda$ map into $U$, $w_\lambda$ factors through~$U$.
Furthermore, since $\bar\phi$ is a homomorphism of commutative $S$-groups, we can calculate
\[ w_\lambda(x_1, x_2, s) = \bar v_\lambda(x_1^{-1}x_2^{-1}, s) \cdot \bar v_\lambda(x_1, s) \cdot \bar v_\lambda(x_2, s) \]
on $\ov\GG_m \times \ov\GG_m \times S$.

Therefore, applying \cref{coeffs-congruence-lemma} to $\bar v_\lambda$ and $w_\lambda$, we obtain that there exists an $S$-morphism $v'_\lambda \colon \GG_m(\epsilon^2) \times S \to U$ such that
\begin{equation} \label{eqn:vij-bound}
\bar v_\lambda \equiv v'_\lambda|_{\ov\GG_m \times S} \bmod \pi^{2\lambda}.
\end{equation}
Since $\bar u = u_\lambda \cdot \bar v_\lambda$, it follows from \eqref{eqn:vij-bound}
and \cref{formal-group-congruence-piU} that
\[ \bar u \equiv u_\lambda \cdot v'_\lambda|_{\ov\GG_m \times S} \bmod \pi^{2\lambda}. \]
Hence, taking $u_{\lambda+1} = u_\lambda \cdot v'_\lambda \colon \GG_m(\epsilon^2) \times S \to U$, we obtain
\[ \bar u \equiv u_{\lambda+1}|_{\ov\GG_m \times S} \bmod \pi^{2\lambda}. \]
Since $2\lambda \geq \lambda+1$, it follows that \eqref{eqn:ui-congruence} holds for $\lambda+1$.
\end{proof}

The following lemma extracts the main calculation used in the proof of \cref{induction}.
Note that the conclusion of this lemma is stronger than we need: we only need a congruence modulo $\pi^{\lambda+1}$ rather than $\pi^{2\lambda}$.

\pagebreak % COSMETIC

\begin{lemma} \label{coeffs-congruence-lemma}
Let $\bar v \colon \ov\GG_m \times S \to U$ be a morphism of rigid $S$-spaces such that:
\begin{enumerate}[label=(\roman*)]
\item there exists an $S$-morphism $w \colon \GG_m(\epsilon) \times \GG_m(\epsilon) \times S \to U$ satisfying
\[ w(x_1, x_2, s) = \bar v(x_1^{-1}x_2^{-1}, s) \cdot \bar v(x_1, s) \cdot \bar v(x_2, s) \]
for all $(x_1, x_2, s) \in \ov\GG_m \times \ov\GG_m \times S$;
\item the image of $\bar v$ is contained in $U[\epsilon^\lambda]$, where $\lambda \in \ZZ_{\geq 1}$.
\end{enumerate}

Then there exists an $S$-morphism $v' \colon \GG_m(\epsilon^2) \times S \to U$ such that 
\[ \bar v \equiv v'|_{\ov\GG_m \times S} \bmod \pi^{2\lambda}. \]
\end{lemma}

\begin{proof}
Write the Laurent series for $\zeta_i \circ \bar v$ and $\zeta_i \circ w$, respectively, as
\begin{align*}
   \bar v_i(\xi, s) & = \sum_{j \in \ZZ} \bar v_{ij}(s) \xi^j \in \tatealgebra{\cO^\circ(S)}{\xi, \xi^{-1}},
\\ w_i(\xi_1, \xi_2, s) & = \sum_{j,k \in \ZZ} w_{ijk}(s) \xi_1^j \xi_2^k \in \tatealgebra{\cO(S)}{\xi_1, \xi_1^{-1}, \xi_2, \xi_2^{-1}}.
\end{align*}
Since $w_i$ maps $\GG_m(\epsilon) \times \GG_m(\epsilon) \times S$ into~$D$, we have
\begin{equation} \label{eqn:wijk-bound}
\abs{w_{ijk}}_S \leq \epsilon^{\abs{j}+\abs{k}}
\end{equation}
for all $j,k \in \ZZ$.

By hypothesis~(ii) and \cref{formal-group-F2-congruence}, we obtain the following congruence of morphisms $\ov\GG_m \times \ov\GG_m \times S \to D$:
\begin{equation} \label{eqn:w-u-congruence}
w_i(\xi_1, \xi_2, s) \equiv \bar v_i(\xi_1^{-1}\xi_2^{-1}, s) + \bar v_i(\xi_1, s) + \bar v_i(\xi_2, s) \bmod \pi^{2\lambda}.
\end{equation}
Expanding as Laurent series in $\tatealgebra{\cO^\circ(S)}{\xi_1, \xi_2, \xi_1^{-1}, \xi_2^{-1}}$ and comparing the coefficients of $\xi_1^{-j}\xi_2^{-j}$, we obtain the following congruence for all $j \in \ZZ$:
\begin{equation} \label{eqn:w-u-diag-congruence}
w_{ijj} \equiv \bar v_{i,-j} \bmod \pi^{2\lambda}.
\end{equation}

Define $v'_i$ to be the truncation of the Laurent series of $\zeta \circ \bar u$:
\[ v'_i(\xi) = \sum_{j=-\lambda}^\lambda \bar u_{ij} \xi^j. \]
Since $v'_i$ is a Laurent polynomial, it defines a holomorphic function on $\GG_m(\epsilon^2) \times S$.
Since all non-zero terms of $v'_i$ satisfy $\abs{j} \leq \lambda$, and using \eqref{eqn:w-u-diag-congruence} and~\eqref{eqn:wijk-bound}, we obtain that the non-zero coefficients of $v'_i$ satisfy $\abs{\bar v_{ij}}_S \leq \max\{ \abs{w_{i,-j,-j}}_S, \epsilon^{2\lambda} \} \leq \epsilon^{2\abs{j}}$.
Hence, $v'_i$ maps $\GG_m(\epsilon^2) \times S$ into~$D$.
Hence there is a rigid $S$-morphism $v' \colon \GG_m(\epsilon^2) \times S \to U$ such that $\zeta_i \circ v' = v'_i$ for all~$i$.

Thanks to~\eqref{eqn:wijk-bound}, we have $w_{i,-j,-j} \equiv 0 \bmod \pi^{2\lambda}$ whenever $\abs{j} \geq \lambda$.
Therefore, by~\eqref{eqn:w-u-diag-congruence}, $\zeta_i \circ \bar v \equiv v'_i|_{\ov\GG_m \times S} \bmod \pi^{2\lambda}$ for each~$i$, proving the lemma.
\end{proof}

\subsection{Conclusion of proof}

We are now ready to complete the proof of \cref{hom-extends-thm}.

\begin{lemma} \label{exists-u}
There exists a rigid $S$-morphism $u \colon \GG_m(\epsilon) \times S \to U$ such that $u|_{\ov\GG_m \times S} = \bar u$.
\end{lemma}

\begin{proof}
For $i = 1, \dotsc, g$, let $\bar u_i = \zeta_i \circ \bar u \colon \ov\GG_m \times S \to D$.
Then $\bar u_i \in \cO^\circ(\ov\GG_m \times S) = \tatealgebra{\cO^\circ(S)}{\xi, \xi^{-1}}$.
Thus we can write
\begin{equation} \label{eqn:ui-coeffs}
\bar u_i = \sum_{j \in \ZZ} \bar u_{ij} \xi^j
\end{equation}
where $\bar u_{ij} \in \cO^\circ(S)$.
Similarly, for the morphisms $u_\lambda$ given by \cref{induction}, if we let $u_{\lambda i} = \zeta_i \circ u_\lambda \colon \GG_m(\epsilon^2) \times S \to D$, we obtain Laurent series
\[ u_{\lambda i} = \sum_{j \in \ZZ} u_{\lambda ij} \xi^j \]
where $u_{\lambda ij} \in \cO^\circ(S)$ and $\abs{u_{\lambda ij}}_S \leq \epsilon^{2\abs{j}}$ for all $\lambda$, $i$, $j$.

From \eqref{eqn:ui-congruence}, we have $\bar u_{ij} \equiv u_{\lambda ij} \bmod \pi^\lambda$ for all~$\lambda \in \ZZ_{\geq 1}$, $i = 1, \dotsc, g$ and $j \in \ZZ$.
Applying this for $\lambda=2\abs{j}$, along with $\abs{u_{\lambda ij}}_S \leq \epsilon^{2\abs{j}}$, we obtain that $\abs{\bar u_{ij}}_S \leq \epsilon^{2\abs{j}}$ for all~$j \in \ZZ$.
A fortiori, for each~$i$, $\epsilon^{-j}\abs{\bar u_{ij}}_S \to 0$ as $j \to \infty$ so the power series $\bar u_i$ defines a morphism of rigid spaces $u_i \colon \GG_m(\epsilon) \times S \to D$.
The desired $S$-morphism $u \colon \GG_m(\epsilon) \times S \to U$ can now be defined as $\zeta^{-1} \circ (u_1, \dotsc, u_g)$.
\end{proof}

Using the morphism $u$ from \cref{exists-u}, define
\[ \psi = \alpha \cdot u^{-1} \colon \GG_m(\epsilon) \times S \to G. \]
From the definition of $\bar u$ in \cref{baru-through-tube}(a), we obtain that $\psi|_{\ov\GG_m \times S} = \bar\phi$.
Since $\bar\phi$ is an $S$-group homomorphism, the identity principle implies that
\begin{equation} \label{eqn:psi-partial-hom}
\psi(\xi_1, s) \cdot \psi(\xi_2, s) = \psi(\xi_1 \xi_2, s)
\end{equation}
for all $\ov K$-points $(\xi_1, \xi_2, s)$ in $\GG_m(\epsilon) \times \GG_m(\epsilon) \times S$ satisfying $\xi_1 \xi_2 \in \GG_m(\epsilon)$.

For $n \in \ZZ$, define $A_n \subset \GG_m$ to be the annulus
\[ A_n = \{ x \in \GG_m : \epsilon^{n+1} \leq \abs{x} \leq \epsilon^n \}. \]
The sets $A_n$ for $n \in \ZZ$ form an admissible covering of $\GG_m$.
The only non-empty intersections between sets in this covering are
\[ A_{n-1} \cap A_n = \{ x \in \GG_m : \abs{x} = \epsilon^n \}. \]

Define $\psi_n \colon A_n \times S \to G$ by
\[ \psi_n(\xi, s) = \psi(\pi, s)^n \cdot \psi(\pi^{-n} \xi, s), \]
noting that if $\xi \in A_n$, then $\pi^{-n}\xi \in A_0 \subset \GG_m(\epsilon)$.
If $(y,s) \in (A_{n-1} \cap A_n) \times S$, then
\[ \psi_n(y,s) = \psi(\pi,s)^n \cdot \psi(\pi^{-n} y, s) = \psi(\pi,s)^{n-1} \cdot \psi(\pi^{-n+1}, s) = \psi_{n-1}(y,s), \]
where the middle equality follows from \eqref{eqn:psi-partial-hom} since $\pi$, $\pi^{-n}y$ and $\pi^{-n+1}y$ are all in~$\GG_m(\epsilon)$.
Hence we can glue the $\psi_n$ together into a morphism of rigid $S$-spaces $\phi \colon \GG_m \times S \to G$.
Since $\phi|_{\ov\GG_m \times S} = \psi_0|_{\ov\GG_m \times S} = \bar\phi$ is an $S$-group homomorphism, the identity principle implies that $\phi$ is an $S$-group homomorphism.

This completes the proof of \cref{hom-extends-thm}.

\section{Rigid \texorpdfstring{$S$}{S}-groups without boundary} \label{sec:without-boundary}

In this section, we prove \cref{hom-extends-no-boundary}, which is a version of \cref{hom-extends-thm} with condition~(ii) replaced by the condition that the morphism $G \to S$ is without boundary, which can often be more quickly verified in applications.

\subsection{Morphisms without boundary}

The notion of rigid spaces without boundary comes from \cite[Def.~5.9]{Lut90}.
A relative version of this notion was defined by Lamjoun \cite[Def.~1.5]{Lam99}.
It shall be convenient for us to express Lamjoun's definition in terms of another notion, which we call ``morphisms globally without boundary.''

\begin{definition}
A morphism of rigid spaces $X \to S$ is \defterm{globally without boundary} if there exist two affinoid coverings (with the same index set) $\{ U_j \}$ and $\{ W_j \}$ of~$X$, satisfying, for every~$j$, $U_j \Subset_S W_j$.
\end{definition}

\begin{definition} \cite[Def.~1.5]{Lam99}
A morphism of rigid spaces $X \to S$ is \defterm{without boundary} if there exists an affinoid covering $\{ S_i \}$ of~$S$ such that, for every~$i$, $X \times_S S_i \to S_i$ is globally without boundary.
\end{definition}

In the case $S = \Sp(K)$, a morphism $X \to \Sp(K)$ is without boundary if and only if $X$ is a rigid space without boundary in the sense of \cite[Def.~5.9]{Lut90}.

Note that the difference between a morphism of rigid spaces without boundary and Kiehl's definition of a proper morphism of rigid spaces \cite[Def.~2.3]{Kie67} is that for a proper morphism, the coverings $\{U_j\}$ and $\{W_j\}$ of $X \times_S S_i$ must be finite, while for a morphism without boundary, these coverings may be infinite.

The most useful examples of morphisms without boundary are (1) proper morphisms of rigid spaces and (2) the rigid analytifications of morphisms $\fX \to \fS$ of separated $K$-schemes of finite type (this follows from \cite[Cor.~5.11]{Lut90} and \cite[Prop.~1.13]{Lam99}, applied to $\fX^\an \to \fS^\an \to \Sp(K)$).

The notion of morphisms ``globally without boundary'' seems to be usually more difficult to work with than that of morphisms without boundary.  But it is convenient for stating the following lemma.

\begin{lemma} \label{without-boundary-U}
Let $X \to S$ be a morphism of rigid spaces which is globally without boundary, where $X$ is quasi-separated.
Then, for every quasi-compact open subspace $U \subset X$, there exists a quasi-compact open subspace $W \subset X$ such that $U \Subset_S W$.
\end{lemma}

\begin{proof}
Let $\{U_j\}$ and $\{W_j\}$ be affinoid coverings of~$X$ such that, for all~$j$, $U_j \Subset_S W_j$.

Suppose we are given a quasi-compact open subspace $U$.
For each~$j$, let $V_j = U \cap U_j$.
Since $X$ is quasi-separated, each $V_j$ is quasi-compact.
Since $U$ is quasi-compact, there is a finite set~$J$ such that $U = \bigcup_{j \in J} V_j$.

Let $W = \bigcup_{j \in J} W_j$.
This is quasi-compact open subspace of~$X$.
For each~$j$, we have $V_j \subset U_j \Subset_S W_j \subset W$ so $V_j \Subset_S W$.
Hence, by \cref{relatively-compact-union}, $U \Subset_S W$, as required.
\end{proof}

\subsection{Proof of Corollary~\ref{hom-extends-no-boundary}}

By condition~(ii) of \cref{hom-extends-no-boundary}, $G \to S$ is without boundary.
Therefore, we can choose an affinoid covering $\{S_i\}$ of~$S$ such that each morphism $G \times_S S_i$ is globally without boundary.
It suffices to prove the corollary for the base change to each~$S_i$, since we can glue the resulting homomorphisms $\phi_i \colon \GG_m^g \times S_i \to G \times_S S_i$ in order to obtain $\phi \colon \GG_m^g \times S \to G$.
Thus, in order to prove the corollary, it suffices to assume that $S$ is affinoid and that $G \to S$ is globally without boundary.
We henceforth make these assumptions.

Write $\xi_{1}, \dotsc, \xi_{g} \colon \GG_m \times S \to \GG_m^g \times S$ for the inclusions of the direct factors $\GG_m \to \GG_m^g$ over the base~$S$.
Let
\[ \bar\phi_j = \bar\phi \circ \xi_j|_{\ov\GG_m \times S} \colon \ov\GG_m \times S \to G. \]

By condition~(i) of \cref{hom-extends-no-boundary}, $\bar\phi_j$ factors through~$H$, which has a formal model and hence is quasi-compact.
By \cref{without-boundary-U}, there is a quasi-compact open subspace $X \subset G$ such that $H \Subset_S X$.
Thus $\bar\phi_j$ and~$H$ satisfy condition~(ii) of \cref{hom-extends-thm}.
Meanwhile condition~(i) of \cref{hom-extends-thm} is the same as condition~(i) of \cref{hom-extends-no-boundary}.

Hence, by \cref{hom-extends-thm}, for each $j = 1,\dotsc,g$, there is a unique homomorphism of rigid $S$-groups $\phi_j \colon \GG_m \times S \to G$ whose restriction to $\ov\GG_m \times S$ is equal to $\bar\phi_j$.
Multiplying these $\phi_j$ together, we obtain a unique homomorphism $\phi \colon \GG_m^g \times S \to G$ satisfying $\phi \circ \xi_j = \phi_j$ for all~$j$, completing the proof of \cref{hom-extends-no-boundary}.

\bibliographystyle{amsalpha}
\bibliography{rigid}

\end{document}